\documentclass{amsart}

\usepackage{etex}
\usepackage{amsmath, amssymb}
\usepackage{array}
\usepackage[frame,cmtip,arrow,matrix,line,graph,curve]{xy}
\usepackage{graphpap, color, paralist, pstricks}
\usepackage[mathscr]{eucal}
\usepackage[pdftex]{graphicx}
\usepackage[pdftex,colorlinks,backref=page,citecolor=blue]{hyperref}
\usepackage{pifont}

\newtheorem{thm}{Theorem}[section]
\newtheorem{prop}[thm]{Proposition}
\newtheorem{cor}[thm]{Corollary}
\newtheorem{lem}[thm]{Lemma}
\theoremstyle{definition}

\newtheorem{remark}[thm]{Remark}
\newtheorem{obs}[thm]{Observation}

\newcommand{\gO}{\Omega}
\newcommand{\go}{\omega}

\newcommand{\cO}{\mathcal{O} }

\newcommand{\cC}{\mathcal{C} }

\newcommand{\cE}{\mathcal{E} }

\newcommand{\cG}{\mathcal{G} }

\newcommand{\cI}{\mathcal{I} }
\newcommand{\cJ}{\mathcal{J} }
\newcommand{\Oh}[0]{{\mathcal O}}

\newcommand{\cW}{\mathcal{W} }

\newcommand{\beq}[1]{\begin{equation}\label{#1}}
\newcommand{\enq}[0]{\end{equation}}

\newcommand{\eps}{\epsilon}
\newcommand{\gb}[0]{\beta }

\newcommand{\gD}[0]{\Delta }
\newcommand{\gG}[0]{\Gamma }
\newcommand{\gL}[0]{\Lambda}
\newcommand{\gs}[0]{\sigma}
\newcommand{\gz}[0]{\zeta}

\newcommand{\vp}[0]{\varphi}

\newcommand{\bn}[0]{\bigskip\noindent}
\newcommand{\mn}[0]{\medskip\noindent}
\newcommand{\nin}[0]{\noindent}

\newcommand{\ov}[0]{\bar}
\newcommand{\sub}[0]{\subseteq}
\newcommand{\sm}[0]{\setminus}
\newcommand{\0}[0]{\emptyset}
\newcommand{\ra}[0]{\rightarrow}

\newcommand{\mis}[0]{\mbox{\rm{mis}}}
\newcommand{\MIS}[0]{\mbox{\rm{MIS}}}

\newcommand{\im}{\mbox{\rm{im}}}

\newcommand{\supp}[0]{\mbox{\rm{supp}}}

\begin{document}

\title{The number of maximal independent sets in the Hamming cube}

\author{Jeff Kahn \and Jinyoung Park}
\thanks{The authors are supported by NSF Grant DMS1501962 and BSF Grant 2014290.}
\thanks{JK was supported by a Simons Fellowship.}
\email{jkahn@math.rutgers.edu, jp1324@math.rutgers.edu}
\address{Department of Mathematics, Rutgers University \\
Hill Center for the Mathematical Sciences \\
110 Frelinghuysen Rd.\\
Piscataway, NJ 08854-8019, USA}

\begin{abstract}
Let $Q_n$ be the $n$-dimensional Hamming cube and $N=2^n$.  
We prove that the number of maximal independent sets in $Q_n$ is asymptotically
\[2n2^{N/4},\]
as was conjectured by Ilinca and the first author in connection with a question of Duffus, Frankl and R\"odl.

The value is a natural lower bound derived from a connection between maximal independent sets and induced matchings. The proof that it is also an upper bound draws on various tools, among them “stability” results for maximal independent set counts and old and new results on isoperimetric behavior in $Q_n$.
\end{abstract}

\maketitle

\section{Introduction}

\subsection{Theorems and definitions}

The purpose of this paper is to prove the following statement, which was 
conjectured by Ilinca and the first author \cite{IK} in connnection with 
a question of Duffus, Frankl and R\"odl \cite{DFR}.
We use $\mis(G)$ for the number of maximal independent sets (MIS's) of a graph $G$,
$Q_n$ for the $n$-dimensional Hamming cube and $N$ for $2^n$. 
(A few basic definitions are recalled below.)

\begin{thm}\label{thm:main}
\beq{misqn}
\mis(Q_n) \sim 2n2^{N/4}.
\enq
\end{thm}
\nin
(As usual, $a_n \sim b_n$ means $a_n/b_n \rightarrow 1$ as $n \rightarrow \infty$. The original question from \cite{DFR}, answered in \cite{IK}, just asked for the asymptotics of
$\log\mis(Q_n)$.)  

The general context for Theorem~\ref{thm:main} is asymptotic enumeration in the spirit of,
prototypically, Erd\H{o}s, Kleitman and Rothschild \cite{EKR}, who showed that
a.a.\footnote{that is, all but a $o(1)$ fraction as $n\ra\infty$}
triangle-free graphs on $n$ vertices are bipartite.

Here we typically have some collection $\cC$ 
(really, a \emph{sequence} of collections $\cC_n$) and the goal is 
to say that some natural, easily understood subcollection $\cC_n'$ 
accounts for a.a.\ of $\cC_n$.

Within this broad context
Theorem~\ref{thm:main} is closest to 
a short sequence of results beginning nearly
forty years ago with
the asymptotic solution of
Dedekind's Problem by Korshunov \cite{Korshunov} 
(and 
Sapozhenko \cite{Sap1}).
Other results in the sequence give  asymptotics for the number of independent
sets in $Q_n$ (again Korshunov and Sapozhenko \cite{K-S} and Sapozhenko \cite{Sap2})
and the numbers of proper $q$-colorings of $Q_n$ for $q=3,4$,
due respectively to Galvin \cite{G} and the present authors \cite{KPq}.
(Similar ideas appear in work on certain statistical physics models,
e.g.\ in \cite{GK,PS}, to mention just the earliest and most recent instances.)
The reader familiar with the earlier combinatorial results will note the striking
purity of Theorem~\ref{thm:main}, which involves no terms akin to the powers of $e$
in \cite{K-S,Sap2,G,KPq} or the far messier ``extra" terms in the Dedekind asymptotic.

Before proceeding we briefly recall why the r.h.s.\ of \eqref{misqn}
is an (asymptotic) lower bound.
As usual an \textit{induced matching} (IM) is an induced subgraph that is a matching.
It is easy to see that the largest IM's of $Q_n$ are of size $N/4$ and that there are exactly $2n$
of these, here called \emph{canonical matchings} and denoted $M^*$
(see below for a precise description).  Each $M^*$ gives rise to exactly
$2^{N/4}$ MIS's, gotten by choosing one vertex from each edge of $M^*$ and
extending the resulting independent set to the (unique) MIS containing it.
It is also easy to see (an argument is sketched at the end of this section)
that the overlaps between the sets of MIS's gotten from different $M^*$'s are negligible,
and the lower bound follows.
In analogy with the problems mentioned above (beginning with Dedekind's) one
may think of $2n$ "phases," one for each $M^*$.  
(E.g.\ for the simplest of the earlier instances---independent sets, or,
in physics, the \emph{hard-core model}---the vast majority of those sets 
consist almost entirely of vertices of a single parity, and the phases
are "even" and "odd.")

So what Theorem~\ref{thm:main} is really saying
is that the number of MIS's \emph{not} corresponding
to canonical matchings is negligible.  The proof of this goes roughly as follows.
We first (``Step 1"; Lemma~\ref{lem:step1}) show that almost every MIS is
``associated" with some ``large" IM.
Step 2 (Lemma~\ref{lem:step2}) then says that each ``large" IM is close to some $M^*$.
Finally, in Step 3 (Lemma~\ref{lem:step3}), we show that the number of MIS's that are associated with 
an IM close to
some $M^*$ but are not obtained from $M^*$ as above (that is, miss at least one edge of $M^*$) is small.

Before making this sketch concrete we need a few definitions.  (A few more are given in Section~\ref{sec:MD}.)

\noindent \textit{Definitions.}
We use $Q_n$ for the Hamming cube, the graph on $\{0,1\}^n$ with
two vertices (strings) adjacent if they differ in exactly one coordinate.
(We use $v,w,x,y$ for vertices 
and $vw$ or $(v,w)$ for an edge joining $v$ and $w$.)
A \emph{subcube} is $\{x:x_i=y_i~\forall i\in J\}$ for some $J\sub [n]$ and $y\in \{0,1\}^J$.
Until further notice (in Section~\ref{sec:step3}), we 
use $\cE $ and $\Oh$ for the sets of even and odd vertices of $Q_n$
(where the parity of $x$ is the parity of $\sum x_i$).
The string gotten from $x$ by flipping its $i$th coordinate (the neighbor of $x$ in
\emph{direction} $i$) is denoted $x^i$, and we define the \textit{parity} of the edge $xx^i$ to
be the parity of $\sum_{j \ne i} x_j$.

We use $I$ and $M$ for MIS's and IM's (respectively),
$\cI(G)$ for the set of MIS's in $G$, and, in particular, $\cI$ for $\cI(Q_n)$. Write $I\sim M$ if each edge of $M$ meets $I$.
For bookkeeping purposes we
fix a linear order ``$\prec$" on the set of IM's of $G$
and define
$M_G(I)$ to be the first (in $\prec$) of the largest induced matchings $M$ satisfying
\beq{MsimI}
I\sim M
\enq
and
\beq{Rim'(I)}
\nabla(V(M),I\sm V(M))=\0
\enq
(where $\nabla(X,Y)$ is the set of edges between $X$ and $Y$ and $V(M)$ is the set of vertices contained in edges of $M$).
We also set $m_G(I)=|M_G(I)|$
and abbreviate $M_{Q_n}(I)=M(I)$ and 
$m_{Q_n}(I)=m(I)$.

A \textit{canonical matching} of $Q_n$
is the set of edges $vv^i$ of parity $\eps$, for some
$i \in [n]$ and $\eps \in \{0,1\}$.
Canonical matchings are denoted $M^*$.
It is easy to see that (as mentioned earlier)
the maximum size of an IM is $N/4$, and an IM is of this size iff it is canonical.
We set $\cI^*=\{I \in \cI : I \sim M^* \mbox{ for some } M^*\}$.

Throughout the paper $\log$ is $\log_2$ and $\beta=\log(3/2)~(\approx .58)$.

We can now formalize our plan.  Let
\[\cJ =\{I \in \cI: m(I) >(1-\log^3 n/n)N/4\}.
\]
(The $\log^3 n/n$ is not optimal, but it is convenient and we have some room.)

\begin{lem}\label{lem:step1} 
\[ |\cI \setminus \cJ|=o(2^{N/4}).\]
\end{lem}
\nin (The actual bound will be $\log |\cI \setminus \cJ| <(1-\gO(\log^3 n/n))N/4$.)

\begin{lem}\label{lem:step2}
With $\beta  =\log(3/2)$ (as above), if
\beq{Msize}|M| =(1-o(n^{-\beta}))N/4,\enq
then there is an $M^*$ with
\[|M \gD M^*|=o(N)\]
(equivalently, $|M\cap M^*|= (1-o(1))N/4$).
\end{lem}
\nin We believe Lemma \ref{lem:step2} remains true when $n^{-\beta}$ is replaced by $n^{-1/2}$.
This improvement, which is easily seen to be the best one can hope for here,
would follow from Conjecture ~1.10 of \cite{KPi}; see Remark~\ref{see.remark}.

We use Lemma~\ref{lem:step2} to say that each $I$ not covered by Lemma~\ref{lem:step1}
(i.e. $I\in \cJ$) is closely tied to some $M^*$; precisely, for a suitable $\gz=\gz(n)=o(1)$, each $I\in \cJ$ satisfies
\[
\mbox{there is an $M^*$ with $|M(I) \gD M^*| < \gz N.$}
\]
Thus, the following lemma completes the proof of Theorem \ref{thm:main}.

\begin{lem}\label{lem:step3} For any $M^*$,
\beq{IM*}
|\{I\not\in \cI^*: |M(I)\gD M^*|< \gz N \}|= 2^{N/4-\go(n/\log n)}
\enq
\end{lem}

\nin
\textit{Outline.} 
Lemmas \ref{lem:step1}, \ref{lem:step2} and \ref{lem:step3} are proved in Sections \ref{sec:step1}, \ref{sec:step2} and \ref{sec:step3} respectively. Section \ref{sec:tools} collects a few further basics and the various earlier results we will use, and Section \ref{sec:alg} treats an algorithm that underlies the proofs of Lemmas \ref{lem:step1} and \ref{lem:step3}.
 While they do require additional ideas, Lemma \ref{lem:step1} is substantially based on Theorem \ref{KPs} below, which was proved in \cite{KPs} (and which will also play a crucial role in the proof of Lemma \ref{lem:step3}), and the main point for Lemma \ref{lem:step2} (Theorem \ref{KPii} below) is from \cite{KPi}
(which was originally motivated by the present application).
So the most important and interesting contribution of the present paper is the proof of Lemma \ref{lem:step3},
which depends especially on Sapozhenko's Lemma \ref{lem:GS}.
It is interesting that Theorem \ref{thm:main} seems to require as much as it does \textit{in addition to} Sapozhenko's ingenious and difficult argument for (ordinary) independent sets in $Q_n$ (the main ideas for which are pretty well represented by Lemma \ref{lem:GS}).

We close this section with the promised lower bound discussion.

\begin{prop}\label{lem:lbd} Let $M_1^*, M_2^*$ be distinct canonical matchings and $\cI^*_j=\{I \in \cI : I \sim M^*_j\}$ for $j=1,2$. Then
\[|\cI^*_1 \cap \cI^*_2| \le 3^{N/8}.\]
\end{prop}

\begin{proof}
This is easy and we just give an informal sketch. We may assume $M_1^*$ and $M_2^*$ use different directions, since otherwise $\cI_1^* \cap \cI_2^* =\{\cE, \Oh\}$. We may further assume the two directions are $n-1$ and $n$, and consider the natural projection $\pi:\{0,1\}^{[n]} \rightarrow \{0,1\}^{[n-3]}$; thus the $\pi^{-1}(v)$'s are copies of $Q_3$ partitioning $Q_n$. It is then easy to see that for an $I \in \cI_1^* \cap \cI_2^*$ there are at most three possibilities for each 
$I \cap \pi^{-1}(v)\cap V(M_1^*\cup M_2^*)$ (and that these choices determine $I$),
yielding the bound in the lemma. \end{proof}

\section{Tools}\label{sec:tools}

\subsection{More definitions}\label{sec:MD}
Let $G$ be a graph and $x \in V=V(G)$. As usual, $N_x$ denotes the neighborhood of $x$ and $N(A)=\cup_{x \in A} N_x$. 
For 
$S \subseteq V$, $d_S(x)=|S \cap N_x|$.
For $A \subseteq V$, the \textit{closure} of $A$ is
\[[A]=\{x \in V: N_x \subseteq N(A)\}\]
and $A$ is \textit{closed} if $A=[A]$.

For a positive integer $k$, say 
$A \subseteq V$ is \textit{k-linked} if for any $u,v \in A$, there are vertices $u=u_0, u_1, \ldots, u_l=v$ in $A$ 
such that for each $i \in [l]$, $u_{i-1}$ and $ u_i$ are at distance at most $k$ in $G$. 
The \textit{$k$-components} of $A$ are its maximal $k$-linked subsets. 
(So we use ``component" for a \textit{set of vertices} rather than a subgraph.)
In what follows we will only be interested in $k=2$.

In the rest of the paper we use $V$ for $V(Q_n)$.

For disjoint $A,B \sub V$ and $i \in [n]$, $\nabla (A,B)=\{(x,y) \in E(Q_n) : x \in A, y \in B\}$, $\nabla A = \nabla(A, V \setminus A)$, $\nabla_i(A,B)=\{(x,x^i):x \in A, x^i \in B\}$ and $\nabla_i A= \nabla_i(A, V \setminus A)$.

\subsection{Step 1 supplies}\label{sec:tools1}

(This name is not very accurate, as the main point of the section, Theorem~\ref{KPs},
is crucial for Lemma~\ref{lem:step3} as well as Lemma~\ref{lem:step1}.)

\begin{thm} [Hujter-Tuza \cite{HT}] \label{HT}
For any $m$-vertex, triangle-free graph $G$,
\[\log\mis(G) \le m/2,\]
with equality iff $G$ is a perfect matching.
\end{thm}

\begin{thm}[\cite{KPs}, Theorem 3.4] \label{KPs}
There is $c>0$ such that for any $\eps$ and $m$-vertex, triangle-free graph $G$,
\[\log |\{I \in \cI(G): m_G(I) < (1-\eps) m/2\}|<(1-c\eps)m/2.\]
\end{thm}
\nin
In particular, with $\im(G)$ denoting the size of a largest induced matching in $G$, $\log \mis(G)> (1-\eps)m/2$ implies $\im(G)>(1-O(\eps))m/2$; 
this is Theorem~1.4 of \cite{KPs}, a ``stability" version of Theorem \ref{HT}.

In what follows we will mainly be concerned with $I\in \cI$ (recall this is $\{\mbox{MIS's of $Q_n$}\}$) having $m(I)\approx N/4$,
for which the next little point will be helpful.
\begin{obs}\label{rimObs}
If $|M(I)|> (1-\eps)N/4$, then $|I\sm V(M(I))|< \eps N$.
\end{obs}
\begin{proof}
With $M=M(I)$, $W=V(M)$ and $Z=N(W)\sm W$, we have
$I\cap Z=\0$ (by definition of $M(I)$) and
\[
(n-1)|W| ~=~|\nabla(W,Z)| ~\leq ~(n-1)|Z|,
\]
implying $|Z|\geq |W|$ and
\[
|I\sm V(M)|~\leq~ |V\sm (W\cup Z)| ~< ~\eps N.
\]
\end{proof}

\subsection{Step 2 supplies}

For $A \subseteq V$, define $h_A:V \rightarrow \mathbb N$ by
\[ h_A(x) = \begin{cases} d_{V \setminus A}(x) &\mbox{ if } x \in A, \\ 0 &\mbox{ if } x \notin A. \end{cases} \] 
For $f:V \rightarrow \mathbb N$ and a probability measure $\nu$ on $V$,
\[ \int f d \nu := \sum_{x \in V} f(x)\nu(x).\]
In the next three results, the second and third of which are derived from the first in \cite{KPi}, 
$\mu$ is uniform measure on $V$.

\begin{thm} [\cite{KPi}, Theorem 1.1] \label{KPi} For any $A \subseteq V$,
\[ \int h_A^\beta d\mu \ge 2 \mu(A)(1-\mu(A)).\]
\end{thm}

\begin{cor} [\cite{KPi}, Corollary 3.2] \label{KPi.cor}
If $R \cup S \cup U$ is a partition of $V$ with $\mu(R \cup U)=\alpha$, then
\[\int_R h_{R \cup U} d\mu \ge 2\alpha(1-\alpha)-n^\beta \mu(U).\]
\end{cor}

\begin{thm} [\cite{KPi}, Theorem 1.9] \label{KPii}

Suppose $A \cup B \cup W$ is a partition of $V$ with $\mu(A)=(1\pm \eps)/2$, $\mu(W) \le \eps n^{-\beta}$ and
\[ |\nabla(A,B)|<(1+\eps)2^{n-1}.\]
Then there is $i \in [n]$ such that
\[|\nabla_i A|=(1-O(\eps))2^{n-1}.\]
Furthermore, there is a subcube $C$ 
(of dimension $n-1$)
such that
\[\mu(C \Delta A)=O(\eps).\] 
\end{thm}
\nin
(As usual $a \pm b$ denotes a quantity from $(a-b,a+b)$.)

\begin{remark}\label{see.remark}
Conjecture 1.10 of \cite{KPi} says Theorem \ref{KPii} remains true if we replace $\beta$ by $1/2$; this would imply the
strengthening of Lemma \ref{lem:step2} mentioned earlier.
\end{remark}

\subsection{Step 3 supplies}
Recall that a \textit{composition} of $m$ is a sequence $(a_1,\ldots,a_s)$ of positive integers summing to $m$ (the $a_i$'s are the \textit{parts} of the composition), and that:

\begin{prop}\label{prop:comp} 
The number of compositions of $m$ is $2^{m-1}$ and the number with at most $b \le m/2$ parts is $\sum_{i\le b}{m-1 \choose i} < \exp_2[b\log(em/b)]$.
\end{prop}

We use the next proposition in bounding the numbers of certain types of 2-linked sets in $Q_n$. It follows from the fact (see e.g. \cite[p.\ 396, Ex.11]{Knuth}) that the
infinite $\Delta$-branching rooted tree contains precisely
\[\frac{{{\Delta n} \choose n}}{(\Delta-1)n+1} \le (e\Delta)^{n-1}\]
rooted subtrees with $n$ vertices.

\begin{prop} [\cite{G}, Lemma 1.6] \label{prop:setcost} 
For each fixed $k$, the number of $k$-linked subsets of $V$ of size $x$ containing some specified vertex is at most $2^{O(x\log n)}$.
\end{prop}

The next two results are standard(ish) isoperimetric inequalities for $Q_n$; see e.g. \cite[Lemma 1.3]{K-S} or \cite[Claim 2.5]{GS} for the first and \cite[Lemma 3.4]{KPq} for the second.

\begin{prop} \label{prop:isop}
For $A \subseteq \cE$ with $|A| \le N/4$,
\[\frac{|N(A)|-|A|}{|N(A)|}=\gO(1/\sqrt n).\]
\end{prop}

\begin{prop}\label{prop:isop1} 
For $A$ a subset of either $\cE$ or $\cO$ and $k=n^{o(1)}$,
\[\mbox{if } |A|=n^k, \mbox{ then } |N(A)|>(1-o(1))(|A|n/k).\]
\end{prop}

The next lemma, which recalls what we need from \cite{Sap87}, follows from Lemmas 5.3-5.5 of the more accessible \cite{GS}. Here, for whatever $A \sub \cE$ is being discussed, we take $G=N(A)$ and $t=|G|-|[A]|$. 

\begin{lem} \label{lem:GS}
For $g \in [n^4, N/4]$ and $\cG=\cG(a, g):=\{A \subseteq \cE : A \mbox{ is 2-linked and closed, } |A|=a \mbox{ and } |G|=g\}$, there are $\cW=\cW(a, g) \subseteq 2^{\cE} \times 2^{\cO}$ with
\[ |\cW|=2^{O(t\log^2n/\sqrt n)}\]
and $\varphi=\varphi_{a,g}:\cG \rightarrow \cW$ such that for each $A \in \cG$, $(S,F):=\varphi(A)$ satisfies:

\begin{enumerate}[(a)]
\item $S \supseteq A~(=[A]), F \subseteq G$;
\item $d_F(u) \ge n-\sqrt n/\log n \quad \forall u \in S$;
\item $|S| \le |F| + O(t/(\sqrt n \log n)).$
\end{enumerate}
\end{lem}
\nin (For the reader familiar with or consulting \cite{GS}: we use the lemmas mentioned above
with $\vp=n/2$ (note his $\vp$ is unrelated to the one in Lemma~\ref{lem:GS})
and $\psi=\sqrt n/\log n$; the restriction to $g \le N/4$, with Proposition~\ref{prop:isop}, gives $t=\gO(g/\sqrt n)$, so that in Lemma 5.4 of \cite{GS} we are looking at the second bound in (20).)

\section{Algorithm}\label{sec:alg}

Here we isolate an algorithmic framework that will play key roles in the proofs of
Lemmas \ref{lem:step1} and \ref{lem:step3}.
Like the basic algorithm in \cite{KPs}, this is motivated by
an idea for counting (ordinary) independent sets due to Sapozhenko \cite{Sap07},
but the analyses here seem new; see the preview at the end of this section.

For the algorithm we fix some order ``$\prec$" on $V=V(Q_n)$.
(This basic discussion makes sense for a general graph
$G$ and independent set---as opposed to MIS---$I$, but we stick to what we will use.)

\mn
\textbf{[Algorithm] }
Given $I\in \cI$ and $W \sub V$, let $X_0=W$ and repeat for $i=1, 2, \ldots$:

\begin{enumerate}
\item[(1)] Let $x_i$ be the first (in $\prec$) vertex of $X_{i-1}$
among those with largest degree in $X_{i-1}$.

\item[(2)] If $x_i \in I$ then let $X_i=X_{i-1} \setminus (\{x_i\} \cup N(x_i))$; otherwise, let $X_i=X_{i-1} \setminus \{x_i\}$.  Set $\xi_i=\textbf{1}_{\{x_i \in I\}}$.

\item[(3)]  STOP:  the stopping rule will vary.
\end{enumerate}

Let $X=X(I)$ be the final $X_i$ and $H=H(I)=Q_n[X]$.
Notice that $\xi=\xi(I)=(\xi_1, \xi_2, \ldots)$
encodes a complete description of the run of the algorithm (so we may also write $H=H(\xi)$),
including, in particular, the identities of the $x_i$'s; also that
\beq{algnote1}\mbox{$\xi(I)$ determines $X$ and $I \setminus X$}\enq
and
\beq{algnote2}\mbox{$I \cap X$ is an MIS of $H$}.\enq

Analyses for the several uses of \textbf{[Algorithm]} below will vary.
We close this discussion of what's common with two easy observations that will be needed in all cases, together with the promised preview.

\begin{prop}\label{Pstupid}
For $\xi$ running over binary strings, with $|\xi|$ denoting the length of $\xi$,
and positive integers $l$ and $r\leq l/2$,
\[
\log|\{\xi:|\xi|\leq l, |\supp(\xi)|\leq r\}| ~\leq ~ r\log(l/r) +O(r)+\log(l+1).
\]
\end{prop}
\begin{proof}  This follows from $\log \sum_{t\leq r}{l \choose t}\leq lH(r/l)$ (where $H$ is
binary entropy).\end{proof}

\begin{prop}\label{PZ}
If $Z\sub W\sub V$, $d_Z(x)\leq d ~\forall x\in Z$ and 
$|\nabla W|\leq L$, then
\[
|Z|\leq (2n-d)^{-1}(n|W|+L).
\]
\end{prop}

\begin{proof}  This follows from
\[
n|W\sm Z|~\geq ~|\nabla(Z,W\sm Z)| ~\geq ~|Z|(n-d) -L.
\]
\end{proof}

\nin
\emph{Preview}

In our uses of \textbf{[Algorithm]} one reason for stopping will usually be that degrees in $X_i$ fall below
some specified $d$; we then have a \emph{tradeoff:}

\nin
(i)  \emph{Larger} $d$ tends to mean smaller $\supp(\xi)$:  each $x_i\in I$ removes at least $d$ 
vertices from consideration, so $|\supp(\xi)|< |W|/d$.
(And by Proposition~\ref{Pstupid}, smaller $\supp(\xi)$ means fewer possibilities for $\xi$.)

\nin
(ii)  \emph{Smaller} $d$ tends to mean smaller $X$ (by Proposition~\ref{PZ}, applied with $Z=X$).  
Note the effect of varying $d$
is not insignificant here since we are usually interested in $|X|-|W|/2$.

\nin
A simple but seemingly new idea that is one of the main drivers of the present work
is that we can do better in (i) if we lower bound $d_{X_{i-1}}(x_i)$, not by the 
final cutoff $d$, but by whatever we get by plugging $X_{i-1}$ in for $Z$ in Proposition~\ref{PZ}.
We give two implementations of this idea; the first, in Section~\ref{sec:step1}, is more elegant and precise,
while the cruder version in Section~\ref{sec:step3} more simply illustrates the basic principle.
(See also Remark~\ref{Remark6.5}.)

\section{Proof of Lemma \ref{lem:step1}}\label{sec:step1}

In this section, $I$ is always in $\cI \setminus \cJ$.
The eventual key here is Theorem~\ref{KPs}, but we need to first reduce to a
place where the theorem is helpful---so to a vertex set of size not much more than
$N/2$ since we are interested in induced matchings of size around $N/4$.
The algorithm of Section \ref{sec:alg} provides a ``cheap'' way to do this.

For any subgraph $H$ of $Q_n$, let 
\[\MIS^*(H)=\{I \in \cI(H): m_H(I) ~\le~(1-\log^3 n/n)N/4\},\]
and $\mis^*(H)=|\MIS^*(H)|$. (Note the cutoff for $m_{H}(I)$ here is the one in the definition of $\cJ$.)

For the proof of Lemma \ref{lem:step1} we run \textbf{[Algorithm]} with input our unknown
$I$, stopping as soon as either
\begin{enumerate}
\item $|\supp(\xi)| \ge \frac{\log n}{2n}N$, or 
\item $X_i=\emptyset$,
\end{enumerate}
and let $X=X(I)$ and $H=H(I)$ ($=H(\xi)$) be as in Section \ref{sec:alg}.
Notice that $I \in \cI \setminus \cJ$ implies
\[\mbox{$I \cap X \in \MIS^*(H(I))$,}\]
so
\[|\cI \setminus \cJ| \le \sum_\xi \mis^*(H(\xi))\]
(where the sum runs over possible $\xi$'s). Proposition \ref{Pstupid} bounds the number of possible $\xi$'s by
\[\exp_2\left[O\left(\log^2 n/n\right)N\right],\]
so that Lemma \ref{lem:step1} will follow from
\beq{mistar}\log\mis^*(H(I)) \le \left(1-\gO\left(\frac{\log^3 n}{n}\right)\right)N/4 ~\mbox{ for all }~ I.\enq

\noindent \textit{Proof of} (\ref{mistar}). Fix $I$ and let $X=X(I)$ ($=V(H(I))$). 
We first show that $|X| $ 
cannot be much larger than $N/2$. Let $d_i=\max\{d_{X_i}(v):v \in X_i\}$ and $\ov X_i = V \setminus X_i$.

\begin{obs}\label{prop:step1}
For each $i$, $|X_i| \le (1+d_i/n)N/2$.
\end{obs}

\begin{proof} This follows from Proposition \ref{PZ} with $Z=X_i$ and $W=V$ (and $L=0$).\end{proof}

Define $\alpha_i$ by
\[|X_i|=(1+\alpha_i)N/2~ ;\]
so $\alpha_0=1$ and Observation \ref{prop:step1} says
\beq{sharp}d_i \ge \alpha_in.\enq

\begin{obs}\label{prop:step1(2)}
If $\xi_i=1$, then $\alpha_i < (1-2n/N)\alpha_{i-1}$.
\end{obs}

\begin{proof}
Using (\ref{sharp}), we have
\[(1+\alpha_i)N/2=|X_i|=|X_{i-1}|-d_{i-1}-1 <(1+\alpha_{i-1})N/2-\alpha_{i-1}n,\]
and the observation follows.
\end{proof}

\begin{prop}\label{Xsmall}
\beq{xstarubd}|X|<(1+1/n)N/2.\enq

\end{prop}

\begin{proof}
Let $\alpha$ be the final $\alpha_i$ (so $|X|=(1+\alpha)N/2$). Assuming (as we may) that $X \ne \emptyset$, we have
\[|\supp(\xi)| \ge \frac{\log n}{2n}N,\]
so that Observation \ref{prop:step1(2)} (with $\alpha_0=1$ and the fact that $\alpha_i$ is decreasing in $i$) gives
\[\alpha\le(1-2n/N)^{\frac{\log n}{2n}N}<1/n ,\]
which is (\ref{xstarubd}). \end{proof}

\noindent The results quoted in Section \ref{sec:tools1} combined with (\ref{xstarubd}) now easily give (\ref{mistar}): if
\[|X|<(1-\gO(\log^3 n/n))N/2\]
then (\ref{mistar}) follows from Theorem \ref{HT}; otherwise, applying Theorem \ref{KPs} with $m=|X|$ and a suitable $\eps=\gO(\log^3 n/n)$ gives
\[\log\mis^*(H)< (1-c\eps)|X|/2< (1-\gO(\log^3 n/n))N/4.\]

\section{Proof of Lemma~\ref{lem:step2}} \label{sec:step2}

Let $M$ be as  in Lemma \ref{lem:step2}.
We may assume that
\beq{ldir} \mbox{$n-1$ and $n$ are the two directions least used by $M$.}\enq
Let $\pi:V \rightarrow V(Q_{n-2})$ be the natural projection, namely
\[\pi( (\eps_1, \ldots, \eps_n))=(\eps_1,\ldots,\eps_{n-2}),\]
and for $v \in V(Q_{n-2})$, let 
\[U_v=\pi^{-1}(v)=\{(v, \eps_{n-1}, \eps_n) : \eps_{n-1}, \eps_n \in \{0,1\}\}.\]
For the rest of this section, ``measure'' refers to $\mu$, the uniform measure on $V(Q_{n-2})$.

Say $v\in V(Q_{n-2})$ is \emph{red} (or \emph{in} $R$) if
$U_v\cap V(M)=\{(v, 0, 0), (v,1,1)\}$ and \textit{blue} ($v$ in $ B$)
if $U_v \cap V(M)=\{(v,1,0), (v, 0, 1)\}$.
(So $v\not\in R\cup B$ iff $U_v$ either contains an edge of $M$ or meets $V(M)$ at most once.)
Say $v\in R\cup B$ is \emph{good} if there is a (necessarily unique) $v'\in N_v$ with
the same color ($R$ or $B$) as $v$; thus
$v$ is good iff $U_v$ meets two edges of $M$ and these have the same direction, and 
\beq{obsgood2}
\mbox{if $w\sim v$ are both good then they have the same color iff $w=v'$.}
\enq

Let $X$ be the set of good vertices and $W=V(Q_{n-2}) \setminus X$ (the set of ``bad'' vertices).

\begin{obs} \label{obs:W}$\mu(W)=o(n^{-\beta})$
\end{obs}

\begin{proof}
As already noted, $v$ is bad iff it satisfies one of: (i) $U_v$ contains an edge of $M$;
(ii) $|U_v \cap V(M)|\le1$;
(iii) $v$ is red or blue and there is no vertex of the same color in $N_v$.
It follows from (\ref{ldir}) that the fraction of $v$'s of the first type is $O(1/n)$, and from
\eqref{Msize} that the fraction of the second type is $o(n^{-\beta})$.

For $v$ as in (iii), let $xy$ be one of the two $M$-edges meeting $U_v$, say with $x\in U_v$
and $y\in U_w$.
Then $U_w\cap V(M)=\{y\}$, $w$ is as in (ii), and $v$ is the unique vertex of $Q_{n-2}$
for which $U_v$ and $U_w$ are connected by an edge of $M$.  Thus the number of vertices in (iii) is less
than (actually at most half) the number in (ii), so these too make up an $o(n^{-\gb})$-fraction of the whole.
\end{proof}

Recall that the parity of the edge $vv^i$ is the parity of $\sum_{j\neq i}v_j$ and notice that
\beq{obsgood6}
\mbox{$v$ and $vv^i$ have the same parity iff $v_i=0$.}
\enq

It follows from \eqref{obsgood2} that 
$
T:=\{(v,v'):v\in X\}
$ 
is a perfect matching of $Q_{n-2}[X]$.
\begin{obs}\label{obsgood3}
Each $e=vv'\in T$ corresponds to two edges of $M$ 
($((v,0,0),(v',0,0))$ and $((v,1,1),(v',1,1))$ if $v\in R$ and similarly if $v\in B$),
and these edges have the same parity as $e$ if $v\in R$ and the opposite parity if $v\in B$.
\end{obs}

Let $\Gamma=Q_{n-2}[X]-T$. 

\begin{obs}\label{diffcomp} For each $e \in T$, the ends of $e$ are in different components of $\Gamma$.
In particular no component of $\Gamma$ has measure more than 1/2.
\end{obs}

\begin{proof}
Assume for a contradiction that $e=xy$ and 
$P=(x=x_0,x_1,x_2,\ldots,x_k=y)$ is a path in $\Gamma$. 
Notice that (\ref{obsgood2}) implies $x_i$ and $x_{i+1}$ have different colors, 
while $x$ and $y$ have the same color. 
Thus $P \cup \{e\}$ is an odd cycle in $Q_{n-2}$, which is impossible.
\end{proof}

For the rest of this discussion we do not distinguish between components and their vertex sets.

\begin{prop}\label{2partn}
$\Gamma$ contains two components of measure $1/2-o(1)$.
\end{prop}

\nin
(We really only need one such component, but for the same price can give the correct picture.)

\begin{proof}
This follows from Observation~\ref{diffcomp} and
\beq{Z}
\mbox{If $Z$ is a union of components of $\Gamma$ with $z:=\mu(Z)\leq 1/2$, 
then $z$ is either $o(1)$ or $1/2-o(1)$.}
\enq
\begin{proof}[Proof of \eqref{Z}]
Set $Y=X \setminus Z$. Since $\nabla(Z,Y) \subseteq T$ and $T$ is a perfect matching of $Q_{n-2}[X]$, we have $h_{Z \cup W}(x) \in \{0,1\}$ for $x \in Z$, which with Corollary~\ref{KPi.cor} 
(applied in $Q_{n-2}$ with $(R,S,U)=(Z,Y,W)$) and Observation~\ref{obs:W} gives
\[z \ge \int_Z h_{Z\cup W} d\mu \ge 2z(1-z)-o(1),\]
implying \eqref{Z}.
\end{proof}\end{proof}

Let $Z$ be one of the two large components promised by Proposition \ref{2partn} and 
$Y=X \setminus Z$. 
Again (as in the proof of \eqref{Z}), we have $h_{Z \cup W}(x) \in \{0,1\}$ for $x \in Z$, which with 
Observation~\ref{obs:W} and Theorem~\ref{KPii} implies that there are $i \in [n-2]$ with
\beq{nablaiA}
|(\nabla_i Z)\cap T|=|\nabla_i (Z,Y)| \sim 2^{n-3}
\enq
and $\eps \in \{0,1\}$ such that
\[
\mbox{all but $o(2^n)$ vertices of $Z$ lie in 
the subcube $C(i,\eps)=\{v:v_i=\eps\} ~~~(\subseteq V(Q_{n-2}))$.}
\]

Assume (w.l.o.g.) that $\eps=0$ and set 
\[
Z'=\{v\in Z\cap C(i,0): vv^i\in T\}.
\]

Connectivity of $Z$ and \eqref{obsgood2} imply 
\beq{obsgood5}
\mbox{any two vertices of $Z$ either agree in both color and parity or disagree in both.}
\enq

Finally, for Lemma~\ref{lem:step2}:
For $v,w\in Z'$, Observation~\ref{obsgood3} and \eqref{obsgood6} imply that 
the edges of $M$ corresponding to $vv^i$ and $ww^i$ have the same 
parity iff $v$ and $w$ either agree in both parity and color or disagree in both;
but \eqref{obsgood5} says this is true for \emph{any} $v,w\in Z'$.
So \emph{all} edges of $M$ corresponding to edges of $\nabla_i(Z',Y)$ have the same parity
and the lemma follows from \eqref{nablaiA}.

\section{Proof of Lemma \ref{lem:step3}} \label{sec:step3}

For the discussion in this section we fix a canonical matching $M^*$ and show
(proving Lemma~\ref{lem:step3})
\beq{IM*}
|\{I\not\in \cI^*: |M(I)\gD M^*|< \gz N\}|= 2^{N/4-\go(n/\log n)}.
\enq
Assume (w.l.o.g.) that $M^*$ is the set of odd edges in direction $n$ and let
$\pi:V(Q_n) \rightarrow V(Q_{n-1})$ be the projection
\[\pi((\eps_1, \ldots, \eps_n))=(\eps_1,\ldots,\eps_{n-1}).\]
Thus $\pi(V(M^*))$ is the set of odd vertices in $Q_{n-1}$, which we from now on denote by
$\Oh$.

For $\eps \in \{0,1\}$ let $V_\eps=\{x \in V(Q_n):x_n=\eps\}$, and for $v \in V(Q_{n-1})$
let $\pi^{-1}(v)=\{v_0, v_1\}$ where $v_\eps \in V_\eps$.
(We will not use the \emph{coordinates} of $v$, so "$v_\eps$" should cause no confusion.)
For $I\in \cI$, define the \emph{labeling $\gs=\gs(I)$ of $V(Q_{n-1})$} by:
\[ \sigma_v=
\begin{cases}
0 & \mbox{ if } \quad v_0 \in I\\
1 & \mbox{ if } \quad v_1 \in I\\
\Lambda & \mbox{ if } \quad I \cap \{v_0, v_1\}=\emptyset\\
\end{cases}\]
Say $v$ is \textit{unoccupied} if $\sigma_v = \Lambda$,
and \textit{occupied} otherwise. Note that (since $I \in \cI$)
\beq{gs1}
\mbox{no two adjacent vertices have the same label from $\{0,1\}$}
\enq
and
\beq{gs2}
\mbox{if $\gs_v=\Lambda$ then both $0$ and $1$ appear on neighbors of $v$.}
\enq
Call a labeling $\gs:V(Q_{n-1}) \rightarrow \{0,1,\Lambda\}$ \textit{legal} if it satisfies 
\eqref{gs1} and \eqref{gs2},
and notice that
$I\mapsto \gs(I)$ is a bijection between $\cI$ and the set of legal labelings.
We will find both viewpoints useful in what follows
and will assume, often without explicit mention,
that when we are discussing $I$ the labeling referred to is $\gs(I)$.

\mn

For the rest of Section~\ref{sec:step3} we restrict to $I$ as in \eqref{IM*}, noting that then $\gs=\gs(I)$
satisfies
\beq{as1}
 \mbox{all but a $o(1)$-fraction of odd vertices are occupied} 
 \enq
and, by Observation~\ref{rimObs},
\beq{as2}
\mbox{only a $o(1)$-fraction of the even vertices are occupied.}
\enq

Notation below ($\cE^*$, $A_i$ and so on) is for a given $I$, which the notation suppresses.
Write $\cE^*$ for the set of occupied even vertices.
 Notice that $I \notin \cI^*$ implies that
there is at least one unoccupied $v\in \Oh$, which by \eqref{gs2} must have neighbors in both
$\gs^{-1}(0)$ and $\gs^{-1}(1)$;
in particular
\beq{as3}
\mbox{there is a non-singleton 2-component in $\cE^*$.}
\enq
(Recall $k$-components were defined in Section~\ref{sec:MD}.)

\nin
\emph{Notation.}

\begin{itemize}
\item $A_i$'s : non-singleton 2-components of $\cE^*$

\item $A=\cup A_i$
\item $G_i=N(A_i)$, $G=N(A)$
\item $A_i \mbox{ (or simply $i$) is }
\begin{cases} \mbox{\textit{ small } if $|G_i|<n^4$ and} \\
\mbox{\textit{ large } otherwise} \end{cases}$

\item $\hat X=\pi^{-1}(X)~$ (for $X \subseteq V(Q_{n-1})$).

\end{itemize}

We usually (without comment) use lower case letters for the cardinalities of the sets
denoted by the corresponding upper case letters, \emph{except} that we use
$a$ for $|[A]|$ and $a_i$ for $|[A_i]|$.
(Recall the closure $[A]$ of $A$ was defined in Section \ref{sec:MD}.)  We also set
$t_i=g_i-a_i$ and $t=g-a$, noting that $a\geq \sum a_i$ 
($[A]$ \emph{can} properly contain $\cup [A_i]$), so $t\leq \sum t_i$.

Before moving to lemmas we record two basic observations.  
The first says that in some sense all the action is in the $[A_i]$'s and $G_i$'s
(though this only approximately describes what will happen in the main argument; see \eqref{allverts}).
\beq{allverts'}
\mbox{All vertices of
$\Oh\sm G$ are occupied.}
\enq
\begin{proof}  All neighbors of the set in \eqref{allverts'} are in
$\cE\sm A$,
and any occupied vertex
from \emph{this} set is a singleton 2-component of $\cE^*$,
so by \eqref{gs2} has all its neighbors occupied (with a common label). \end{proof}

The second observation (this will be crucial; see \eqref{no.edge}-\eqref{rim'1} and 
\eqref{long.obs}, which leads \emph{via} \eqref{exactG} to
\eqref{35'}) is
\beq{occ.nbr}
\mbox{for each $i$, each edge contained in $\hat{G}_i$ has a neighbor in $I\cap \hat{A}_i$}
\enq
(that is, one of its ends has such a neighbor; note these edges form an induced matching in $Q_n$).

\subsection{Main lemma}

\nin 
We continue to restrict to $I$ as in \eqref{IM*} and to suppress dependence on $I$ in our notation.
In what follows we use ``cost of X'' for the $\log$ of the number of possibilities for X.

Before turning to our main point, Lemma~\ref{lem:step3main},
we observe that there is not much to do when $g$ is large: 
\begin{lem}\label{lem:step3main1} 
The number of $I$'s with $g=\gO(N)$ is $2^{N/4-\gO(N)}.$
\end{lem}
\begin{proof}
By (\ref{as2}), the cost of specifying $A$ is at most $\log {N/4 \choose \le o(N)} = o(N)$,
and that for labeling $A$ is at most $|A|=o(N)$.
But $A$ and its labels  determine $G$ and \emph{its} labels, while
\eqref{allverts'} says that the cost of labeling $\Oh\sm G$ (given $G$)
is at most $N/4-g$ and that the labels for $\Oh\sm G$ determine those for $\cE\sm N(G)$ (and 
all labels on $N(G)\sm A$ are $\gL$).  The lemma follows.\end{proof}

We may thus assume from now on that (say)
\beq{gsmall}
g< N/4,
\enq
so that, by Proposition~\ref{prop:isop},
\beq{tlarge}
\mbox{$t=\gO(g/\sqrt{n})~$ and $~t_i=\gO(g_i/\sqrt{n})~$ for each $i$.}
\enq
This small but crucial point will be used 
repeatedly in what follows; indeed, one may say that the purpose
of Lemmas 1.2 and 1.3 was to get us to \eqref{tlarge}.
(Namely:  Lemmas 1.2 and 1.3 lead to \eqref{as2}; \eqref{as2} is the 
basis for Lemma~\ref{lem:step3main1}; and Lemma~\ref{lem:step3main1}
allows us to restrict to \eqref{gsmall}, where we have \eqref{tlarge}.)

\begin{lem}\label{lem:step3main} For any $a\neq 0$ and $g<N/4$
\beq{gOt}
\log|\{I: |[A]|=a, |G|=g\}|=N/4-\go(t/\log n).
\enq
\end{lem}

To see that this (with Lemma~\ref{lem:step3main1}) gives Lemma~\ref{lem:step3}, note that
we always have $g\geq 2n-2$, and that if $g \le n^2$ (say) then Proposition~\ref{prop:isop1} 
gives $t\sim g$.  Thus Lemma~\ref{lem:step3main} and \eqref{tlarge} bound the number of $I$'s 
satisfying \eqref{gsmall} by 
\[
\mbox{$2^{N/4}\left[ n^42^{-\go(n/\log n)}+ \sum_{g>n^2}g2^{-\go(g/(\sqrt{n}\log n))}\right]$}
= 2^{N/4-\go(n/\log n)} 
\]
(where the irrelevant $n^4$ and initial $g$ in the sum are for
choices of $(g,a)$ and $a$ respectively).

\subsection{Proof of Lemma~\ref{lem:step3main}}

Before beginning in earnest, we dispose of the minor cost of specifying the $a_i$'s and $g_i$'s
(with $\sum a_i\leq a$, $\sum g_i=g$).
The only thing to notice here is that, since $g_i\geq 2n-2$ $\forall i$,
the number of $i$'s is less than $g/n$.
Thus Proposition~\ref{prop:comp} bounds the cost
of the $g_i$'s by $(g/n)\log (en)$
and that of the $a_i$'s by 
\[
\left\{\begin{array}{ll}
(g/n)\log (en)&\mbox{if $(g>)~ a > 2g/n$,}\\
2g/n&\mbox{if $a \le 2g/n$,}
\end{array}\right.
\]
so also the overall ``decomposition" cost by 
\beq{pf:decomp}
O(g\log n/n)=O(t\log n/\sqrt n).
\enq

\mn
\emph{Preview and objective}

It remains to specify $A_i$'s (and thus $G_i$'s and $[A_i]$'s) corresponding to the above parameters, and
a labeling ($\gs$) compatible with these specifications.
For small $i$'s it turns out to be easy to directly identify the $A_i$'s and their labels
(which also gives the associated $G_i$'s and $[A_i]$'s and \emph{their} labels).

For the large $i$'s we think of
``identification" and ``labeling" phases, roughly corresponding to identifying the $[A_i]$'s 
(and $G_i$'s),
and then the restriction of $\gs$ to these sets---``roughly" because
in the most interesting (``slack") case the
first phase will not actually succeed in identifying the $[A_i]$'s.
The identification phase
takes place in the projection on $Q_{n-1}$
and leans mainly on Lemma~\ref{lem:GS}.
For the labeling phase
we return to $Q_n$ and work with maximal independent
sets rather than labelings (recall these are interchangeable),
with arguments again based on the algorithm
of Section~\ref{sec:alg}.
It is here that the crucial role of $\cJ$ will finally appear.

The large $i$'s will be of two types, ``tight" and ``slack."
The slack $i$'s are treated last, when we already have full information on the
small and tight $i$'s.  Here we
produce a single pair $(S,F)\sub \cE\times \Oh$ satisfying
(\emph{inter alia;} e.g.\ the role of $F$ will appear later)
\beq{Ssupseteq}
S\supseteq \cup\{[A_i]: \mbox{$i$ slack}\}
\enq
and
\[
S\cup N(S) ~\text{is disjoint from} ~\cup\{[A_i]\cup G_i:\mbox{$i$ small or tight}\},
\]
and then specify labels for $S\cup N(S)$. 

Since $N(S)\supseteq \cup\{G_i:\mbox{$i$ slack}\}$, \eqref{allverts'} gives
\beq{allverts}
\mbox{all vertices of
$\Oh\sm (\cup\{G_i:i ~\text{small or tight}\} \cup N(S))$ are occupied.}
\enq
Note also that
\[
\mbox{a (legal) labeling is determined by its restriction to $\cup\{[A_i]:i \mbox{ small or tight}\} \cup S \cup \Oh$,}
\]
since each $v$ not in this set (so $v\in \cE$) has at least one occupied neighbor
(for if all neighbors of $v$ are unoccupied, then $v$ is occupied and, by \eqref{gs2}, $N^2(v)$ contains an occupied vertex,
so $v$ must be in some $A_i$).

Thus the cost
of $\gs$ given its restriction to
\[
\cup\{[A_i]\cup G_i:\mbox{$i$ small or tight}\} \cup S\cup N(S)
\]
(so in particular the identity of this set) is at most
\beq{benchmark}
N/4 - \left[\sum\{g_i:\mbox{$i$ small or tight}\} +|N(S)|\right].
\enq

\nin
This gives us a benchmark: 
for Lemma~\ref{lem:step3main},
the cost of the above information (through specification of labels for
$S\cup N(S)$) should be less by $\omega(t/\log n)$ than the subtracted quantity in \eqref{benchmark}
(which in particular makes the decomposition cost \eqref{pf:decomp} negligible).
In the event, this will hold fairly locally: we will wind up paying
$g_i-\gO(t_i)$ for each small or tight $i$ 
and $|N(S)|-\omega(t'/\log n)$ for (all) the slack $i$'s, where $t'=\sum\{t_i:\mbox{$i$ slack}\}$.
(We will repeat this last bit more precisely at the end of the section, following the proof of Lemma \ref{lem:alg2}.)

\mn
\emph{Small $i$'s.}
As suggested above, these are easy. Since $|A_i| \le a_i~(=|[A_i]|)$, the cost of identifying $A_i$, together with its labels,
is at most
\beq{cost:Asmall}
(n-2)+O(a_i \log n) +a_i~ =~ n +O(a_i \log n) 
~<~ g_i - (1/2-o(1))t_i.
\enq
Here the first two terms on the l.h.s., representing the cost of identifying $A_i$, are given
by Proposition~\ref{prop:setcost}, and the final bound follows from 
$g_i\geq \max\{ 2n-2,t_i\}$ and $a_i=O(g_i/n)$, the latter holding for small $i$ by
Proposition~\ref{prop:isop1}.

But $A_i$ and its labels determine $G_i$, $[A_i]$ and their labels
(the labels since all vertices of
$N(G_i) \sm A_i$ are labeled $\gL$); so \eqref{cost:Asmall} actually bounds the total
cost of identifying \emph{and} labeling $[A_i]\cup G_i$.

\bn
\emph{Large $i$'s.}
For a given large $i$,
Lemma~\ref{lem:GS} gives
$\cW=\cW(a_i,g_i)$, $\vp=\vp_i$, $S=S_i$ and $F=F_i$ (as in the lemma), at cost
$O(t_i\log^2n/\sqrt n)$; so the cost of specifying these for all large $i$ is
\beq{cost:approx} 
\mbox{$O(\sum t_i\log^2 n/\sqrt n).$}\enq

Let $\eps=\eps_n$ be a parameter satisfying
\beq{eps}
1 \gg \eps \gg 1/\log n,
\enq
and say $i$ is \textit{tight} if
(with $\eps$ as in \eqref{eps})
\beq{bigS}
g_i-f_i \le \eps t_i
\enq
and \textit{slack} otherwise. (As usual we use $s_i=|S_i|$ and $f_i=|F_i|$. 
The role of $\eps$ is just to enable proper definitions of "tight" and "slack.")

For our purposes the most significant difference between these two possibilities
is that specification of $([A_i], G_i)$ given $(S_i, F_i)$ is cheap if $i$ is tight,
but becomes unaffordable as the difference in \eqref{bigS} grows;
this leads to the following plan.
We first treat tight $i$'s, in each case paying for the full specification of $[A_i]$
(which determines $G_i$)
and then the labels of $[A_i]\cup G_i$.

We then combine and slightly massage the remaining (slack) $S_i$'s and $F_i$'s,
taking account of what we know so far, to produce a single pair $(S,F)$ that in some sense
approximates the slack parts of the configuration, and from $(S,F)$ go directly
to specification of labels (so we learn---implicitly---the 
identities of the slack $[A_i]$'s and $G_i$'s only when we learn their labels.)

\mn
\emph{Tight $i$'s.} 
The next two lemmas bound the total cost of a tight $i$
(so of $[A_i]$, $G_i$ and their labels) by
\beq{cost:tight}
g_i-\gO(t_i).
\enq

\begin{lem}\label{Ptight1}
For tight i, the cost of $([A_i], G_i)$
given $(S_i, F_i)$
is $o(t_i)$.
\end{lem}

\begin{lem}\label{lem:alg2'}
The cost of labeling a given $[A_i]\cup G_i$
is 
$g_i-\gO(t_i)$.
\end{lem}
\nin
\emph{Remark.}  Lemma \ref{lem:alg2'} does not require that $i$ be tight.

\begin{proof}[Proof of Lemma \ref{Ptight1}]
Given $(S_i,F_i)$, fix some $A^*\in \varphi^{-1}(S_i,F_i)$. (Note $A^*$ is closed.
Note also that we are not considering possibilities for $A^*$, just naming a particular
choice associated with $(S_i,F_i)$---e.g.\ the first member of $\vp_i^{-1}(S_i,F_i)$ according to some
order---so the specification costs nothing.  This strangely helpful device is from \cite{EKRII}.)
The key (trivial) point here is that (given $G^*:=A^*$)
\[
\mbox{$(G^* \sm G_i, G_i \sm G^*)$ determines $(G_i,[A]$).}
\]
So we should bound the costs of
$G^* \sm G_i$ and $G_i\sm G^*$.  Since $G^*\sm G_i\sub G^*\sm F_i$,
the cost of $G^* \sm G_i$ is at most $|G^*\sm F_i|\le \eps t_i=o(t_i)$
(since $i$ is tight).

On the other hand,
\[
G_i\sm G^* = N([A_i]\sm A^*)\sm G^*
\]
(since each $x\in G_i\sm G^*$ has a neighbor in $[A_i]$ and none in $A^*$);
so we may specify $G_i\sm G^*$ by specifying a $Y\sub [A_i]\sm A^*\sub S_i\sm A^*$
of size at most $|G_i\sm G^*|\leq g_i-f_i=o(t_i)$ with $G_i\sm G^* = N(Y)\sm G^*$
(let $Y$ contain one neighbor of $x$ for each $x\in G_i\sm G^*$).
But, since $s_i< f_i+o(t_i) \leq g_i+o(t_i)$ (see (c) of Lemma~\ref{lem:GS}), we have
$|S_i\sm A^*|=s_i-a_i\leq t_i+o(t_i)$; and the cost of specifying a subset of size 
$o(t_i)$
from a set of size $O(t_i)$ is $o(t_i)$.
\end{proof}

\begin{proof}[Proof of lemma \ref{lem:alg2'}]
As promised earlier (see the discussion following \eqref{pf:decomp}) we now return to $Q_n$ and,
with $W=\widehat{[A_i]} \cup \hat G_i$,
bound the number of MIS's in $\gG:= Q_n [W]$.
(Note that since $A_i$ is a 2-component of $\cE^*$,
$I\cap W$ is an MIS in $\Gamma$, possibilities for which
correspond to possible (legal) labelings of $[A_i]\cup G_i$).

We run \textbf{[Algorithm]} (of Section~\ref{sec:alg})
twice (or, really, once with a pause; here we index steps by $j$
since $i$ is already taken).
For the first run (on all of $\gG$, with input the unknown $I$) we STOP as soon as
\[
d_{X_j}(x)\le n^{2/3} \mbox{ for all $x \in X_j$}.
\]
This implies $|\supp(\xi)|\leq 2(g_i+a_i)n^{-2/3}$
(note e.g.\ $|\hat{G_i}|=2g_i$), so Proposition~\ref{Pstupid} bounds the cost of this run by
\beq{run.cost1}
(2+o(1))(g_i+a_i)n^{-2/3}\log(n^{2/3}) =o(t_i),
\enq
where the "$o(t_i)$" uses \eqref{tlarge}.
On the other hand, with $Z_1$ the final $X_j$ from this run, Proposition~\ref{PZ} with $Z=Z_1$, $d=n^{2/3}$ and
\beq{firstL}
L =|\nabla(W)| = 2(n-1)(g_i-a_i)
\enq
gives
\begin{eqnarray}
|Z_1| & \le & (2n-n^{2/3})^{-1}(2n(g_i+a_i)+2(n-1)(g_i-a_i))\nonumber\\
&<& (2n-n^{2/3})^{-1}4n g_i ~<~ (1+n^{-1/3})2g_i.\label{Z1bd}
\end{eqnarray}

We next run \textbf{[Algorithm]} on $Q_n[Z_1]$ and STOP as soon as either
\begin{itemize}
\item[(a)] $d_{X_j}(x)\le n^{1/3} \mbox{ for all $x \in X_j$}$ or
\item[(b)] $|X_j|\leq 2a_i$.
\end{itemize}
(Note we treat this as a fresh run rather than a continuation, and recycle
$X_j$ and $\xi$.)

Let $Z_2$ be the final $X_j$ for this run.
From \eqref{Z1bd} and (b) we have
$   
z_1-z_2 \le 2t_i+2n^{-1/3}g_i,
$   
so in view of (a), 
\[
|\supp(\xi)|\le (z_1-z_2)n^{-1/3} \leq 2t_in^{-1/3} + 2g_in^{-2/3} =:r.
\]

\nin
Proposition~\ref{Pstupid} (with this $r$ and $l=|W|\leq 4g_i$) then
bounds the run cost by
\beq{run.cost2}
O((t_in^{-1/3} +g_in^{-2/3})\log n + \log g_i) =o(t_i),
\enq
with the $o(t_i)$ given by \eqref{tlarge}.

Finally we consider the cost of specifying $I\cap Z_2$
(an MIS of $Q_n[Z_2]$).  If the second run ends with $|Z_2|\leq 2a_i$ (as in (b)),
then Theorem~\ref{HT} bounds this cost by
\[
a_i=g_i-t_i.
\]

Suppose instead that the algorithm halts due to (a).
In this case we again use Proposition~\ref{PZ}, now with $Z=Z_2$, $d=n^{1/3}$ and
$L$ as in \eqref{firstL}, to obtain (\emph{cf.} \eqref{Z1bd})
\beq{Z2bd}
|Z_2| <  (1+n^{-2/3})2g_i = 2g_i +o(t_i).
\enq
We now apply Theorem \ref{KPs} in $\gG:=Q_n[Z_2]$.
The key here is \eqref{occ.nbr}, which implies
\beq{no.edge}
\mbox{no edge of $\hat{G}_i$ can belong to $M_\gG(I\cap Z_2)$}
\enq
(since the neighbor promised by \eqref{occ.nbr} cannot come from $I\sm Z_2$,
which has no neighbors in $Z_2$).
It follows that
\beq{rim'1}
m_\gG (I \cap Z_2) \le a_i
\enq
(each edge of $M_\gG (I \cap Z_2) $ meets (possibly meaning equals) one of
the $a_i$ edges of $\hat{A}_i$ and, since  $M_\gG (I \cap Z_2) $ is an induced matching,
the edges met are distinct).
The combination of \eqref{Z2bd}, \eqref{rim'1} and Theorem~\ref{KPs} now again bounds
the cost of $I \cap Z_2$ by $g_i-\gO(t_i)$.

Summarizing, the cost of the two runs of \textbf{[Algorithm]} is $o(t_i)$
(see \eqref{run.cost1}, \eqref{run.cost2}) and, regardless of how these end,
the cost of $I\cap Z_2$ is $g_i-\gO(t_i)$.
The lemma follows.
\end{proof}

\begin{remark}\label{Remark6.5}
Note---\emph{cf.}\ the preview at the end of Section~\ref{sec:alg}---the above argument does not
work if we run \textbf{[Algorithm]} just once, stopping when degrees in $X_j$ fall below $n^{1/3}$;
for our bound on $|\supp(\xi)|$ then becomes $2(g_i+a_i)n^{-1/3}$, so the cost bound in \eqref{run.cost1}
increases to $\Theta(g_in^{-1/3}\log n)$, which need not be small compared to $t_i$.
\end{remark}

\nin
\emph{Slack i's.}
At this point we have found and labeled 
\[
Y:=\cup\left\{[A_i]\cup G_i:\mbox{$i$ small or tight}\right\},
\]
so are left with the slack $i$'s.
As suggested above, these differ from tight $i$'s in that the step that identifies the 
$([A_i],G_i)$'s is no longer affordable, and we instead go directly from the $(S_i,F_i)$'s
to the labeling phase.

Set $Y_\cE=Y \cap \cE$ and $Y_\Oh=Y \cap \Oh$
(so $Y_\cE=\cup\left\{[A_i]:\mbox{$i$ small or tight}\right\}$ and similarly for $Y_\Oh$).
Writing $\cup^s$ and $\sum^s$ for union and sum over slack $i$'s, set
\[
S=(\cup^s S_i) \setminus N(Y_\Oh), ~F=\cup^s F_i, ~
X=N(S) \setminus F
\]
(note $N(Y_\Oh)\supseteq Y_\cE $),
$g'=\sum^sg_i$ and $t'=\sum^s t_i$. Notice that
\beq{gft} g'-f >\eps t'\enq
and that with these definitions we still have the appropriate versions of (a)-(c) of 
Lemma~\ref{lem:GS}, namely:
\begin{itemize}
\item[(a$'$)] $S \supseteq \cup^s[A_i], F \subseteq \cup^sG_i$;
\item[(b')] $d_F(u) \ge n-1-\sqrt n/\log n \quad \forall u \in S$;
\item[(c$'$)] $|S| \le |F| + O(t'/(\sqrt n \log n))$.
\end{itemize}
Here (b') is immediate from the corresponding statement for the $(S_i,F_i)$'s,
as is (c$'$) once we observe that the $F_i$'s are disjoint (since the $G_i$'s are,
and $F_i\sub G_i$).
Similarly, (a$'$) holds because $S_i\supseteq [A_i]$ ($\forall i$) and---the least
uninteresting point here---$N(Y_\Oh)\cap (\cup^s[A_i])=\0$ 
(since there are no
edges between $[A_i]$ and $G_j$ if $i\neq j$).

\mn

The last ingredient in the proof of Lemma~\ref{lem:step3main} is Lemma~\ref{lem:alg2} below,
before turning to which we need a few further observations.

First, we are about to return to $Q_n$ (as in the proof of Lemma~\ref{lem:alg2'}), 
where we will be running \textbf{[Algorithm]} on 
\beq{W}
W:=\hat{S}\cup \hat{F},
\enq
and for use in Proposition~\ref{PZ} will need a bound on $|\nabla W|$.
Setting $\psi =\sqrt{n}/\log n$ (and for the moment still working in $Q_{n-1}$), we have (from (b'))
\beq{s leakage}
|\nabla S \setminus \nabla( S,  F)| \le s\psi\enq
and
\beq{f leakage}\begin{split}|\nabla F \setminus \nabla( S,  F)| &\le f(n-1)-s(n-1-\psi)\\
&=(f-s)(n-1)+s\psi,\end{split}\enq
whence (now in $Q_n$)
\beq{upshot}
L:=|\nabla W|\leq 2(f-s)(n-1)+4s\psi.
\enq

Set $U=\hat S \cup \widehat{N(S)}$.
A second---crucial---observation is
\beq{cru.obs}
\mbox{\emph{$I \cap U$ is an MIS of $Q_n[U]$}.}
\enq
\begin{proof} 
Suppose instead that $x\in U\sm (I\cup N(I\cap U))$.
Then, since $I$ is an MIS of $Q_n$, there are $y\sim x$ and $z\sim x^n$ with 
$y,z\in I$ and $y\not\in U$.
Note this implies $\pi(x)\in N(S)$ (as opposed to $S$),
since otherwise $N(x)\sub U$.
Now $\pi(y)$, $\pi(z)$ are distinct occupied neighbors of $\pi(x)$
(distinct since $y$ and $z$, being in $I$, cannot be adjacent),
meaning that $\pi(x)\in G_i$ for some slack $i$ (slack because $N(S)\cap Y_\Oh=\0$);
but since $A_i$ is a 2-component of $\cE^*$, this implies $\pi(y)\in A_i$ and
$y\in U$, a contradiction.\end{proof}

Finally, we observe that
\beq{long.obs}
\mbox{the edges in $\widehat{N(S)}$ with neighbors in $I\cap U$ are precisely those in $\cup^s\hat{G}_i$.}
\enq
(We have already noted in \eqref{occ.nbr} that
edges in $\cup^s\hat{G}_i$ do have such neighbors (in $\cup^s\hat{A}_i$),
so what \eqref{long.obs} really says is that the remaining edges in $\widehat{N(S)}$ do not.
This is because there are no occupied vertices in $S\sm \cup^s A_i$:
by (b') each $v$ in $S$ has a neighbor in $F$,
so in some slack $G_i$, so if occupied must lie in $A_i$.)
Of course at this point we don't know the $G_i$'s, but what we \emph{can} use from 
\eqref{long.obs} is
\beq{exactG}
\mbox{exactly $g'$ edges in $\widehat{N(S)}$ have neighbors in $I\cap U$ (so in $I\cap \hat{S}$).}
\enq

\begin{lem}\label{lem:alg2}
The cost of labeling $S\cup N(S)$
is at most 
\beq{f+x}
f+x-\gO(\eps t') ~~(=|N(S)|-\gO(\eps t'))
\enq
\end{lem}
\nin
(where $x$ is the size of $X$, which was defined two lines before \eqref{gft}).

\begin{proof}
This is similar to the proof of Lemma~\ref{lem:alg2'}.
We again run \textbf{[Algorithm]} in two stages, but this time only on 
$W$ (defined in \eqref{W}).
As before we STOP the first run when
\[
d_{X_i}(x) \le n^{2/3} ~~\forall x \in X_i,
\]
and let $Z_1$ be the (final) $X_i$ produced by this stage.
We then run the algorithm on $Q_n[Z_1]$, in this case 
stopping as soon as either
\begin{enumerate}
\item[(a)] $d_{X_i}(x) \le n^{1/3}$ for all $x \in X_i$ or
\item[(b)]\label{bee} $|X_i|\leq 2(f-t')$
\end{enumerate}
(of course (b) is possible only if $f\geq t'$), and letting $Z_2$ be the final $X_i$.

As before: the $\xi$ produced by the first run has ($|\xi|\leq |W|=2(s+f)$ and)
$|\supp(\xi)|\leq 2(s+f)n^{-2/3}$, so Proposition~\ref{Pstupid} bounds the cost of this run by
\beq{run.cost1'}
(2+o(1))(s+f)n^{-2/3}\log (n^{2/3})=o(\eps t')
\enq

\nin
(using $s+f \le 2g'$, as follows from (c$'$) and \eqref{gft}, with \eqref{tlarge} and \eqref{eps});
Proposition~\ref{PZ} with $Z=Z_1$, $d=n^{2/3}$ and
$L$ as in \eqref{upshot} gives
\begin{eqnarray}
|Z_1| &<& (2n-n^{2/3})^{-1}[2n(s+f)+2(f-s)n+4s\psi]
\nonumber\\
&=& (2n-n^{2/3})^{-1}[4nf+4s\psi]\nonumber\\
&\leq & 
2f(1+n^{-1/3}) + O(s\psi/n)\nonumber\\
&=&2f(1+n^{-1/3})+o(\eps t')\label{Z1'bd}
\end{eqnarray}
(using $s \psi/n=O(g'/(\sqrt n \log n))=o(\eps t')$, 
which follows from \eqref{tlarge} and \eqref{eps}; this is the reason for the lower bound in \eqref{eps});
(a), (b) and \eqref{Z1'bd}, now with the $\xi$ from the second run, imply 
\[
|\supp(\xi)|\le (z_1-z_2)n^{-1/3} \leq r:= 
\left\{\begin{array}{ll}
(2fn^{-1/3} + O(t'))n^{-1/3} &\mbox{if $f\geq t'$,}\\
(2+o(1))t'n^{-1/3} &\mbox{if $f< t'$;}
\end{array}\right.
\]

Proposition~\ref{Pstupid} with this $r$ and $l=|W|=O(f)$ 
(note (b') implies $s< (1+o(1))f$)
bounds the run cost by
\beq{run.cost2'}
O((fn^{-2/3} +t'n^{-1/3})\log n + \log f) =o(\eps t'),
\enq

\nin
with the $o(\eps t')$ given by \eqref{tlarge} (and $f\leq g'$);
and 
Proposition~\ref{PZ}, with $Z=Z_2$, $d=n^{1/3}$ and, again,
$L$ as in \eqref{upshot}, gives (\emph{cf.}\ \eqref{Z1'bd})
\begin{eqnarray}
|Z_2| &<& (2n-n^{1/3})^{-1}[2n(s+f)+2(f-s)n+4s\psi]
\nonumber\\
&\leq &
2f(1+n^{-2/3})+o(\eps t')~=~ 2f+o(\eps t')\label{Z2'bd}
\end{eqnarray}
(again---as in \eqref{Z1'bd}---using $s \psi/n=o(\eps t')$).

Let $P= I\cap (W\sm Z_2)$ (the set of vertices that were "processed" in the two runs of 
the algorithm and turned out to be in $I$),
$X'=\hat X \setminus N(P)$, $Z'=Z_2\cup X'$ and $\gG=Q_n[Z']$. 
So we are down to identifying $I\cap Z'$
($Z'$ being the set of vertices of $U$ whose membership in $I$ is still in question).
Noting that
\beq{Noting}
\mbox{$I \cap Z'$ is an MIS of $\gG$}
\enq

\nin
(see \eqref{cru.obs}) and recalling that the run costs in \eqref{run.cost1'} and \eqref{run.cost2'} 
were $o(\eps t')$, we find that 
Lemma~\ref{lem:alg2} will follow from
\beq{last}
\mbox{the cost of identifying $I \cap Z'$
is at most
$
f+x-\gO(\eps t').
$}
\enq
(Note we are still 
enforcing \eqref{exactG}.)

If $|Z'| \le 2(f+x)-\gO(\eps t')$ then \eqref{last} is given by 
Theorem \ref{HT} (and \eqref{Noting}).
In particular this is true if the second run ends because of (b),
since then 
$|Z'| \le z_2+2x \le 2(f+x-t')$.

So we are left with cases where the run is stopped by (a) and
\[
|Z'| >2(f+x)-o(\eps t'),
\]
which by \eqref{Z2'bd} implies $x'=2x-o(\eps t')$, i.e.
\beq{XhatX'}
|\hat{X}\sm X'|=o(\eps t').
\enq 
But \eqref{exactG} and the fact that each edge of $\hat{F}$ has a neighbor in $I\cap \hat{S}$
imply that exactly $g'-f>\eps t'$ edges in $\hat X$ have neighbors in $I\cap \hat{S}$, which with \eqref{XhatX'}
yields
\beq{35'}
\mbox{$(1-o(1))\eps t'$ edges in $X'$ have neighbors in $Z_2\cap I\cap \hat{S}$.}
\enq

Now let $M=M_\gG(I\cap Z')$. 
According to the definition of $M_\gG$ (see \eqref{Rim'(I)}) no edge as in \eqref{35'} can be in $M$
(\emph{cf.}\ \eqref{no.edge}), so $M$
fails to cover at least one vertex from each of these edges 
(since, $M$ being induced, $V(M)$ meets any edge not in $M$ at most once).
But then $z' \le 2(f+x)+o(\eps t')$
(which follows from \eqref{Z2'bd} and $z'\le z_2+2x$) implies
\[
m_\gG(I\cap Z') =|M| < (2(f+x)-(1-o(1))\eps t')/2 = f+x-\gO(\eps t'),
\]
and a final application of Theorem~\ref{KPs} 
(with the above bound on $z'$) again gives \eqref{last}, 
completing the proof of Lemma~\ref{lem:alg2}. \end{proof}

In sum (making precise the discussion following \eqref{benchmark}), we have paid:

\begin{enumerate}[$\bullet$]
\item
$O(t\log n/\sqrt n)$ for the decompositions of $a$ and $g$
(see \eqref{pf:decomp});

\item
$g_i-\gO(t_i)$ for specification and labeling of $[A_i]$ and $G_i$
for each small $i$  (see \eqref{cost:Asmall}); 

\item
$O(\sum t_i \log^2 n/\sqrt n)$ for the $(S_i,F_i)$'s, $i$ large (see \eqref{cost:approx});

\item
for each tight $i$, $g_i-\gO(t_i)$ for specification and labeling of $[A_i]$ and $G_i$,
given $(S_i,F_i)$ (see \eqref{cost:tight}); 

\item
$|N(S)|-\gO(\eps t')$ for labeling $S\cup N(S)$, given $(S,F)$ (which is determined by the 
$(S_i,F_i)$'s, together with the $G_i$'s for small and tight $i$);
see \eqref{f+x}. 
\end{enumerate}

Finally, the sum of all these cost bounds is at most
\[
\mbox{$|N(S)| + \sum\{g_i:\mbox{$i$ small or tight}\} +O(\sum t_i\log ^2n/\sqrt{n}) 
-\gO\left(\sum\{t_i:\mbox{$i$ small or tight}\}\right) -\gO(\eps t')$,}
\]
which (recalling $t'=\sum\{t_i:\mbox{$i$ slack}\}$, $t\leq \sum t_i$ and $\eps =\go(1/\log n)$)
is at most
\[
|N(S)| + \sum\{g_i:\mbox{$i$ small or tight}\} -\go(t/\log n);
\]
and combining this with the additional cost in \eqref{benchmark} (paid for the 
remaining labels in $\cO$) gives Lemma \ref{lem:step3main}.


\end{document}